\documentclass[reqno]{amsart}

\usepackage{graphicx}
\usepackage{color}
\usepackage{amssymb}
\usepackage{amsmath}
\usepackage{kpfonts}
\newtheorem{theorem}{Theorem}
 \newtheorem{lem}{Lemma}
  \newtheorem{remark}{Remark}

\UseRawInputEncoding
\usepackage[russian, english]{babel}
 \usepackage{setspace}
\def\bd{{\bf d}}
\def\br{{\bf r}}
\def\dfrac#1#2{\displaystyle{#1\over #2}}

\def\bv{{\bf V}}
\def\bV{{\bf V}}

\def\Div{\mbox{div}\,}

\def\bE{{\bf E}}

\begin{document}


\title[ 
 Euler-Poisson equations with velocity-dependent damping]{
The Euler-Poisson equations with velocity-dependent damping 
}

\author{ Olga S. Rozanova}




\maketitle

\begin{center}
{\it Mathematics and Mechanics Department, Lomonosov Moscow State University, Leninskie Gory,
Moscow, 119991, Russian Federation, rozanova@mech.math.msu.su}
\end{center}
\begin{abstract}
We consider repulsive Euler-Poisson equations in both one-dimensional space and in the multidimensional case with radial symmetry, assuming a power-law velocity-dependent damping coefficient. We show that under this assumption, in the one-dimensional case the set of initial data corresponding to a globally time-smooth solution expands, whereas in other dimensions (except for dimension 4) the damping does not influence on improving of smooth properties of the solution. Namely, any arbitrary small perturbation of the steady state blow up, despite of its amplitude of oscillations decays.
\end{abstract}

\section{Introduction}

We study a version of the repulsive Euler-Poisson
equations with damping
\begin{eqnarray}\label{EP}
\dfrac{\partial N }{\partial t} + \Div(N \bv)=0,\quad
\dfrac{\partial \bv }{\partial t} + \left( \bv \cdot \nabla \right)
\bv =\,  \nabla \Phi\, -\,\nu(|\bv|) \,\bv, \quad \Delta \Phi =N-N_0,
\end{eqnarray}
where the the scalar functions $N$ and $\Phi$ are the density and a repulsive  force potential, respectively,   the vector $\bv$ is the velocity, they depend on
the time $t$ and the point $x\in {\mathbb R}^\bd $, $\bd\ge 1$. Here
$N_0> 0$ is the density background, $\nu\ge 0$ is a damping coefficient.

If we denote $\nabla \Phi = -\bE$, and set  $N_0=1$, such that
\begin{eqnarray}
N=1- \Div \bE,\label{n}\end{eqnarray} we can remove $ N $
from \eqref{EP} and
 rewrite it as
\begin{eqnarray}\label{4}
\dfrac{\partial \bv }{\partial t} + \left( \bv \cdot \nabla \right)
\bv = \, - \bE - \nu(|\bv|) \bv,\quad \frac{\partial \bE }{\partial t} + \bv \Div \bE
 = \bV.
\end{eqnarray}
The
physically natural condition $n>0$ dictates ${\rm div}\bE <1$, see \eqref{n}.

We assume
\begin{eqnarray}\label{nu}
 \nu(|\bv|)=\nu_0|\bv|^k, \quad k>0, \quad \nu_0={\rm const}>0.
\end{eqnarray}
In this paper we study the Cauchy problem for \eqref{EP} or  \eqref{4} and our main concern is to study initial data that guarantee a globally smooth
solution.

The Euler-Poisson equations in their various forms are a very popular model both because of their physical importance (they describe processes in plasma, semiconductors, and astrophysics) and because of the intriguing mathematical properties of their solutions.
In their simplified form, when the medium is considered to be pressureless and consists of particles of only one type, the system of Euler-Poisson equations  allows one to obtain qualitative results concerning threshold phenomena for the Cauchy problem, associated with the most accurate possible division of the initial data into those that correspond to globally smooth solutions in time and those that necessarily lead to a loss of smoothness in finite time.

Success is usually ensured by the ability to describe the solution using Lagrangian dynamics, that is, by reducing the study from a system of partial differential equations to a system of ordinary differential equations. This is possible due to the non-strict hyperbolicity of the system, which in this context means the existence of a unique characteristic curve.

The study of models of this type began with the remarkable work \cite{ELT}, which summarized the results available at that time. Since then, interest in these types of models has been continuously maintained (see, for example, the review in recent papers \cite{T21}, \cite{Recent}).

In this paper, we focus on the so-called repulsive version of the Euler-Poisson equations, which describe a medium consisting of negatively charged particles (electrons). Models of this kind describe electron (or cold) plasma, where the pressure term is neglected.
Interactions between particles that lead to energy losses can be modeled in various ways. These can be viscous terms or integro-differential terms that include the interaction potential. However, the most popular method, not only due to its simplicity but also its physical relevance, is to model interactions using a term proportional to the velocity with a certain coefficient (the coefficient of interactions). The structure of this coefficient may vary depending on the model. It is simplest to consider it constant. However, it is also reasonable to consider it a function dependent on velocity, density, kinetic energy, and so on. In each of these cases, the properties of the solutions change. We will be primarily interested in the effect of the presence of dissipative terms on the change in the domain of initial data corresponding to smooth solutions.

We  consider those classes of solutions of the Euler-Poisson equations for which it is possible to reduce the dynamics to a system of ordinary differential equations, namely, the one-dimensional case in space and the case of radial symmetry in space of any dimension.
The Cauchy problem with $C^1$- smooth initial data in this case is well-defined and the formation of singularities are associated with a finite time blow up of the solution or its first derivatives
\cite{Daf16}.

Let us list known results concerning the possibility of choosing a neighborhood of the origin in the space of initial data corresponding to a globally smooth solution (we call it the smoothness domain, sometimes it is also called the subcritical region).

1. If the term corresponding to the energy loss due to particle interaction is zero, then in dimensions 1 and 4 the smoothness region can be found exactly (\cite{ELT}, \cite{RChD20} for $\bd=1$ and \cite{R_Exeptional} for $\bd=4$). For other dimensions, the smoothness region is absent, i.e., for arbitrarily small general perturbations of the zero equilibrium position, the solution loses smoothness in a finite time.

2. If the interaction coefficient is constant, then the smoothness domain expands for any spatial dimension. For the case of a single spatial variable, this domain can be found exactly; that is, necessary and sufficient conditions for the initial data to fall within the smoothness domain are obtained \cite{ELT},\cite{RChD20}. For the case of several variables, only sufficient conditions for falling within the smoothness domain and sufficient conditions for non-falling within it are known \cite{RD_Radial_EP_friction}. In all cases, the smooth solution decays exponentially at large times, along with its derivatives. The rate of this decay is proportional to the magnitude of the constant interaction coefficient.

3. If the interaction coefficient is a density function, then there exists a condition on this function in terms of the convergence of some improper integral such that, for any initial data, a globally smooth solution exists. This result is proved in \cite{R_PhD_density} for the one-dimensional case, but apparently its analogue is also true for the case of any dimension.

In this article, we consider the case of a power-law dependence of the interaction coefficient on the velocity magnitude. Such interaction coefficients are used in mechanics, in the study of turbulent regimes, and in the modeling of non-Newtonian fluids (e.g., \cite{Astarita}). For example, the case $k=1$ corresponds to so-called aerodynamic friction.

From a mathematical perspective, this model turned out to be very unusual. In dimensions 1 (and 4), this interaction has a predictable effect, expanding the smoothness domain. However, in the remaining dimensions, any small perturbation of the zero equilibrium  leads to the formation of a singularity in a finite time. In other words, this interaction has no effect on the smoothness domain, meaning it does not lead to its formation. This is completely unlike the scenario that occurred with a constant interaction coefficient. This is all the more surprising because the amplitude of the oscillations decreases in this case.

In fact, this phenomenon can be explained in general terms. The decrease in the oscillation amplitude follows a power-law pattern, and the larger the exponent $k$, the slower it is.
At the same time, the instability of oscillations for derivatives is related to Floquet theory, which means their exponential growth. Therefore, the power factor cannot suppress the exponential growth of oscillations, but can only slow it down.

The paper is organized as follows. Section \ref{1D} examines the one-dimensional case. Estimates are obtained for the amplitude of velocity oscillations, and the asymptotic behavior of globally smooth solutions and their derivatives is found at large times. Further, a threshold curve is constructed that precisely bounds the smoothness domain (for the case of zero initial velocity perturbation). This is done analytically for small perturbations and numerically for arbitrary perturbations. Section \ref{MultiD} considers the radially symmetric case of several spatial variables and proves a theorem stating that in dimensions other than 1 and 4, an arbitrarily small perturbation of the zero equilibrium position leads to a blow-up in a finite time. Section \ref{Disc} summarizes the results obtained and discusses prospects for possible further research.

This paper makes extensive use of the methodology developed in \cite{R22_Rad} for the case of a zero interaction coefficient and applied in \cite{RD_Radial_EP_friction} for the case of a constant interaction coefficient. It is based on a linearization technique using Radon's lemma \cite{Radon}, \cite{F}. We do not provide all the details to avoid repeating the steps taken in previous works.

\section{1D case}\label{1D}

The system takes the form
\begin{equation}\label{1DUE}
\dfrac{\partial V }{\partial t}  + V \dfrac{\partial V}{\partial x}=- E- \nu(|V|) V,\qquad
\dfrac{\partial E }{\partial t} + V \dfrac{\partial E}{\partial x} - V
=0,
\end{equation}
where $(V,E)=(\bv , \bE)$.

 For  equations (\ref{1DUE}) we consider the Cauchy problem
  \begin{equation*}
     (V(x,0),  E(x,0))= ( V_0(x), E_0(x))\in  C^2( {\mathbb R})
 \end{equation*}

 Let us denote $q=V_x$, $s=E_x$, $q_0(x) = q(x,0)$, $s_0(x)= s(x,0)$. Since the electron density is $ n = 1 - E_x$, (see (\eqref{n})),
  then  only
$s_0<1$ has a physical sense.

Along the characteristics of  (\ref{1DUE}), the quantities $ V, E $ obey
  \begin{equation}\label{EV}
  \dot V= -E-\nu(|V|) V, \quad \dot E = V,
 \end{equation}
and  $q,s$ obey
   \begin{equation}\label{1Dsq}
  \dot q= -s-q^2-a(|V|) q, \quad \dot s = q -q s, \quad a(|V|)=(|V|\nu(|V|))'.
 \end{equation}
 If $\nu(|V|)=- \nu_0|V|^k$, then $a(|V|)=- \nu_0 (k+1)|V|^k$.

 \begin{lem}\label{L1}
  If  $\nu(|V|)=- \nu_0|V|^k$, then along the characteristic $x=x(t)$, starting from $x_0\in \mathbb R$
 \begin{equation}\label{1DV_t}
  V^2(x(t))\le  \left((V_0(x_0)^2+E_0(x_0)^2)^{-\frac{k}{2}} + {\nu_0 k} t\right)^{-\frac{2}{k}},
 \end{equation}
 \begin{equation}\label{1DEV_t}
    \left((V_0(x_0)^2+E_0(x_0)^2)^{-\frac{k}{2}} + {\nu_0 k} t\right)^{-\frac{2}{k}}\le V^2(x(t))+E^2(x(t))\le V_0^2(x_0)+E_0^2(x_0).
 \end{equation}

 \end{lem}

 \proof From system \eqref{EV} we have
\begin{equation*}\label{EV1}
  \dfrac{d}{dt} (V^2+E^2)= -2\nu(|V|) V^2,
 \end{equation*}
 and
 \begin{equation}\label{EVest}
  (V^2+E^2)=  (V_0^2+E_0^2)  -2\int_0^t \, \nu(|V|) V^2 \, d\tau.
 \end{equation}
 It follows
  \begin{equation*}\label{EV2}
   z(t)\le \alpha -2 \nu_0 \int_0^t \, w(z(t)) \, d\tau, \quad z=V^2(x(t)), \quad \alpha=V_0(x_0)^2+E_0(x_0)^2,\quad w(z)=z^{\frac{k}{2}+1}.
 \end{equation*}
Further, similar to the Bihari-LaSalle arguments, we denote $v=\alpha -2 \nu_0 \int_0^t \, w(z(t)) \, d\tau$.
Therefore $ \dfrac{dv}{dt}=-2\nu_0 w(z(t)\ge -2\nu_0 w(v(t))$. Then we separate the variables as
$ -\dfrac{dv}{w(v(t))}\le 2\nu_0 dt $
and obtain
 \begin{equation*}\label{EV3}
   z(t)\le v(t)\le \left(\alpha^{-\frac{k}{2}}+{\nu_0 k}\,t\right)^{-\frac{2}{k}},
 \end{equation*}
 or \eqref{1DV_t}.

 Further, from system \eqref{EV} we also get
 \begin{equation*}\label{EV4}
  \dfrac{d}{dt} (V^2+E^2)\ge -2\nu_0 ( V^2+E^2)^{\frac{k}{2}+1},
 \end{equation*}
integration implies the left hand side of  \eqref{1DEV_t}. The right side of \eqref{1DEV_t} follows from \eqref{EVest}.
 $\Box$

 \subsection{Long time asymptotics}\label{LTA}

 \begin{theorem}{\rm (Long time asymptotics of $V$)} If  $\nu(|V|)= \nu_0|V|^k$, $k=2m$, $m\in \mathbb N$, and the solution to \eqref{1DUE} exists for all $t\ge 0$, then along any characteristic $x=x(t)$ the following asymptotics takes place
 \begin{equation}\label{1DVas}
  V(x(t))\sim t^{-\frac{1}{k}}\,\sum\limits_{n=0}^\infty \Psi_n(t) t^{-n},\quad t\to\infty,
   \end{equation}
   where $\Psi_n$ are periodic,  $n\in\mathbb N$.
 \end{theorem}

 \proof
 From \eqref{EV} we have the following Li\'enard type equation
 \begin{equation*}
  \ddot V+ \nu_0 (k+1) |V|^k \dot V+V=0.
   \end{equation*}
 We replace $V=t^{-\frac{1}{k}}\,v(t),$ with an analytic in the neighborhood of infinity $v(t)$ and obtain
  \begin{equation*}
  \ddot v+ v+\left(\nu_0 (k+1) v^k-\frac{2}{k}\right)\frac{\dot v}{t}-\left(\nu_0 (k+1) v^k-\frac{k+1}{k^2}\right)\frac{ v}{t^2}
 =0.
   \end{equation*}
Assume that $v=\sum\limits_{n=0}^\infty v_n(t) t^{-n} $, where $v_n(t)=\sum\limits_{l=0}^n \Phi_{n,l}(t) t^l $, where   $\Phi_{n,l}(t)$ are periodic. This means that $v=\sum\limits_{n=0}^\infty \Psi_n(t) t^{-n}$, where $\Psi_0= \sum\limits_{j=1}^\infty \Phi_{j,j}$, $\Psi_1= \sum\limits_{j=2}^\infty \Phi_{j,j-1}$,...,$\Psi_n= \sum\limits_{j=n-1}^\infty\Phi_{j,j-n+1}$.

Substitution gives the following equations for $v_n$, $n=0,1,\dots$:

  \begin{eqnarray*}\label{Gn}
 && \ddot v_0+ v_0=0,\\
 && \ddot v_1+ v_1=\frac{2}{k} \dot v_0 \nu_0(k+1)\dot v_0  v_0,\\
 && \cdots\\
 && \ddot v_n+ v_n=\frac{2(n-1)+1}{k} \dot v_{n-1} +P_2(v_j,\dot v_j), \quad j=0,\dots, n-1,
   \end{eqnarray*}
where $P_2(v_j,\dot v_j, \ddot v_j)$ is a second order polynomial of $v_0,\dot v_0, \ddot v_0, \dots, v_{n-1},\dot v_{n-1}$.
Thus, $v_0=A \sin (t+B)$, $A, B$ are arbitrary constants. For $n\ge 1$ functions $v_n$ necessarily contain secular terms $ c_j(t) t^j$, $c_j\ne 0$,  $j=1,\dots,n$, where $c_j$ are trigonometric polynomials of $\sin (M(t+B))$, $\cos (M(t+B))$, $M=1,\dots,j$, generated by the term
$\dot v_{n-1}$ and trigonometric polynomials generated by $P_2$.
However, the multiplier $t^{-j}$ compensates for the growth over time.
$\Box$

 \begin{theorem}{\rm (Long time asymptotics of derivatives)} If  $\nu(|V|)= \nu_0|V|^k$, $k=2m$, $m\in \mathbb N$, and the solution to \eqref{1DUE} exists for all $t\ge 0$, then along any characteristic $x=x(t)$     
 the derivatives $q(t)$
 and $s(t)$ have
 the  asymptotics
 \begin{equation}\label{1Dqas}
  t^{-\frac{\nu_0  A^k a_k}{2}}\, (A\sin(t+B)+O (t^{-1})),\quad t\to\infty,
   \end{equation}
   where $A$ and $B$ are real constants, $a_k=\,\frac{(k+1)!!}{k!!}$,  $|A|\le ({\nu_0 k})^{-\frac{1}{k}}$.
\end{theorem}
\proof Let us linearize \eqref{1Dsq}, using the Radon lemma (e.g. \cite{F}, \cite{Radon}).
Then
\begin{eqnarray}\label{qsQ}
   q(t)=\frac{p_1(t)}{Q(t)},\quad  s(t)=\frac{p_2(t)}{Q(t)},
 \end{eqnarray}
 where $p_1(t),p_2(t), Q(t)$ solve the linear system
\begin{eqnarray}\label{LinQ4}
\left(\begin{array}{c}  \dot Q\\ \dot p_1\\ \dot p_2
  \end{array} \right)=\begin{pmatrix}
   0 & 1 & 0 \\
  0&- a(V)  & -1\\
  0&1 & 0\\
\end{pmatrix} \left(\begin{array}{c}  Q\\ p_1\\p_2
  \end{array} \right),
 \end{eqnarray}
  subject to the initial data $(Q(0),p_1(0), p_2(0))^T=(1,q(0), s(0))^T$.
  Thus,
\begin{eqnarray}\label{Q}
&&Q=p_2+1-s(0),\\
&&\ddot p_2+a(V) \dot p_2+p_2=0.\nonumber 
 \end{eqnarray}
 If we assume that a solution keeps smoothness for all $t>0$, then $Q(t)>0$ for $t>0$ on all characteristic curves.
 The change $p_2=uv$, $v=\exp\left(-\frac12\int_0^t a(V(\tau)) d\tau\right)$, gives the equation for $u$ as
 \begin{eqnarray}\label{u}
&&\ddot u+ (1-\phi(t))\dot u=0, \quad  \phi=\frac14 a^2(V(t))+\frac12 a'(V(t)).
 \end{eqnarray}
  Lemma 1 implies
  \begin{eqnarray*}
&& a(V)= \nu_0 A^k (k+1) t^{-1}  \sin^k (t+B)+o(t^{-1}), \quad t\to\infty.
 \end{eqnarray*}
   Since the integral
  $\int_\delta^\infty \phi(t) dt, \,\delta>0,$ converges, then the solution of \eqref{u} is bounded (see  \cite{Bellman}, Chapter VI, Th.1) and
   \begin{eqnarray*}
&&u = A \sin (t+B) (1-O(t^{-1})),\quad t\to\infty,
 \end{eqnarray*}
 with constants $A,B$.

 From the estimate \eqref{1DV_t}, valid for all $t>0$ and \eqref{1DVas} we see that $|A|\le \left({\nu_0 k}\right)^{-\frac{1}{k}}$.

  Further, it can be readily shown that
  \begin{eqnarray}\label{aVas}
&& a(V)=\nu_0 (k+1) A^k \, \left(\frac{(k-1)!!}{k!!}+ \theta(t)\right) t^{-1}+ o(t^{-1}),\quad t\to\infty,
 \end{eqnarray}
  where $\theta(t)$ is a trigonometric polynomial. Thus, $v\sim  t^{-\frac{\nu_0  A^k a_k}{2}}$, $t\to\infty$,  $a_k=\,\frac{(k+1)!!}{k!!}$, whereas
 the asymptotics of $p_2$ follows. Since due to \eqref{Q} $Q(t)\to 1-s(0)>0$, $t\to\infty$, (recall that $s(0)<1$), then \eqref{qsQ} implies the asymptotics \eqref{1Dqas} for $s$. Since $\dot Q=p_1$, see \eqref{LinQ4}, the same asymptotics is valid for $q$. $\Box$

\begin{remark} Note that according to Stirling's formula $a_k\sim \sqrt{\frac{2k}{\pi}}$ as $k\to\infty$.  Theorems 1 and 2 imply that the decay rate of the oscillations of the solution $(V,E)$ slows down as $k$ increases, but the decay rate of the oscillations of the derivatives $(q,s)$ increases with $k$.
\end{remark}

 \subsection{Threshold curve}\label{TC}

 The damping coefficient \eqref{nu}, taking into account the estimate \eqref{1DV_t}, can be roughly estimated from above by a constant.
 Since even in the case of constant $\nu$, there always exist initial data for which the solution loses smoothness in finite time \cite{RChD20}, such initial data also exist for the case \eqref{nu}. However, with damping, the set of initial data for which a smooth solution exists for all $t$ (the domain of smoothness, the subcritical domain) expands. The problem arises of strictly separating the domain of the initial data into a set whose Cauchy data lead to a globally smooth solution, and a set corresponding to a solution whose derivatives become infinite in finite time. The hypersurface separating these data in four-dimensional space we call {\it a threshold hypersurface}.

 It is impossible to solve this problem analytically, since it would require finding a complete set of first integrals for the system of four equations \eqref{EV}, \eqref{1Dsq}. Numerically, this problem can be  solved quite simply: integrate the system  \eqref{EV}, \eqref{LinQ4} and decide whether $Q$ vanishes or not. Technically, this is much simpler than solving the original system of equations, where one would have to decide whether or not the components blow  up. However, even if we list all possible sets $(E_0, V_0, s_0, q_0)$, it will be difficult to represent the threshold hypersurface in four-dimensional space. Therefore, we limit the initial data and study only the case $V_0=0$. This choice corresponds to a standard laser pulse that can be reproduced experimentally, and numerical experiments are usually conducted with such initial data \cite{CH18}. Thus, instead of a threshold hypersurface, we will talk about {\it a threshold curve}.

Thus, we construct  the threshold curve on the  plane  $(s_0, E_0)$. In the case of a constant damping coefficient, there is no dependence on $E_0$, so the threshold curve is a straight line. For example, in the simplest case, the singularity formation criterion is
\begin{equation}\label{qs}
q_0^2> 1-2 s_0,
\end{equation}
therefore, the threshold curve  on the  plane $(s_0,E_0)$ is $s_0=\frac12$.

Below, for the case of small deviations from the zero equilibrium, we  derive an analytical expression for the threshold curve, and for any initial data $(s_0,E_0)$, we  construct the threshold curve numerically. It turns out that the larger the amplitude of the initial oscillations, the easier it is to suppress the singularity formation process. This is explained by the fact that a high initial velocity initially acts as a large constant damping, but already during the first oscillation, the amplitude drops sharply, and for the remaining time, the oscillations can be considered small.

\begin{theorem} The threshold curve on the plane $(E_0(x_0),s_0(x_0))$, $x_0\in \mathbb R$,   
for the system \eqref{EV}, \eqref{1Dsq}, $\nu(|V|)= \nu_0|V|^k$, $k=2m$, $m\in \mathbb N$, with initial data $(E, V)|_{t=0}=(E_0,0)$ in a small neighbourhood of point $(\frac12,0)$ has the structure
 \begin{equation}\label{deleps}
 \delta=\nu_0   C_k \epsilon^k + o(\epsilon^k), \quad \epsilon\to 0,\quad C_k={\rm const}>0,
 \end{equation}
where $\epsilon=E_0(x_0)$,  $\delta=s_0(x_0)-\frac12$.
\end{theorem}

\proof
Recall that the blow-up occurs when $Q(t)$, part of solution of \eqref{LinQ4}, turns into zero   for some $t>0$.
It satisfies
\begin{eqnarray}\label{QQ}
&&\ddot Q+(k+1)\nu_0 |V|^k \dot Q+Q-1-s_0=0, \quad Q(0)=1,\quad \dot Q(0)=0.
 \end{eqnarray}
 For small $\epsilon$, where $V=\epsilon \sin t+o(\epsilon)$, we can expand $Q$ as
$$Q(t)=Q_0(t)+Q_1(t) \epsilon^k+o(\epsilon^{k}).$$
We substitute this expansion to \eqref{QQ} and take into account the data $ Q_0(0)=1, \dot Q_0(0)=0,$
$ Q_i(0)= Q_i(0)=0,$ $i\in \mathbb N$. Thus, we
obtain
\begin{eqnarray*}\label{Qi}
&&Q_0= 1+(\cos t -1) s_0,\\
&&Q_1=(k+1) \nu_0 s_0 \left(\frac{\sin^{k+3}t}{k+2}- I_k \cos t\right), \quad I_k=\int_0^t \sin^{k+2} \tau \, d\tau.
 \end{eqnarray*}
 For example, $Q_1=\frac38 \nu_0 s_0 (\cos^2 t \sin t +2\sin t-3 t \cos t)$ for $k=2$.
Further, we can see that
\begin{eqnarray*}\label{Qbar}
&&\bar Q=Q_0(t)+Q_1(t) \epsilon^k
 \end{eqnarray*}
obtain its minimum at point $t=\pi$. Thus, we substitute $s_0=\frac12+\delta$, $t=\pi$ to the algebraic equation $\bar Q=0$ and get
\begin{eqnarray*}\label{Qsub}
-2 \delta+
(k+1) \nu_0 (\frac12+\delta)\int_0^\pi \sin^{k+2} \tau \, d\tau\, \epsilon^k=0.
 \end{eqnarray*}
Therefore
\begin{eqnarray*}\label{Qsub1}
\delta=\frac{(k+1)\nu_0 \bar I_k\epsilon^k }{2 (2-(k+1)\nu_k \bar I_k \epsilon^k)}= \frac{k+1}{4}\, \nu_0 \bar I_k \epsilon^k+o(\epsilon^k),\quad \bar I_k=\int_0^\pi \sin^{k+2} \tau \, d\tau>0,
 \end{eqnarray*}
 it implies \eqref{deleps} with $C_k=\frac{k+1}{4}\, \bar I_k$. Note that $C_k$ increases with $k$. $\Box$

\subsection{Numerical results} Below we give a numerical illustration for the threshold curve obtained from   system  \eqref{EV}, \eqref{LinQ4}. For the case $k=2$, Fig.1  presents  this curve for $s_0$ close to $0.5$ (small $\delta>0$) (left) and $s_0$ close to $1$ (right). We can see that Theorem 3 reasonably describes the behavior of threshold curve for small $\delta$.  For intermediate values of $s_0\in (0.6, 0.8)$ the threshold curve is close to a straight line. Recall that for $\nu_0=0$ the threshold curve is the vertical line $s_0=0.5$, any solution for $s_0>0.5$ blows up independently on $E_0$.

\begin{center}

\begin{figure}[htb]

\hspace{-1cm}
\hspace{1cm}
\begin{minipage}{0.4\columnwidth}
\includegraphics[scale=0.3]{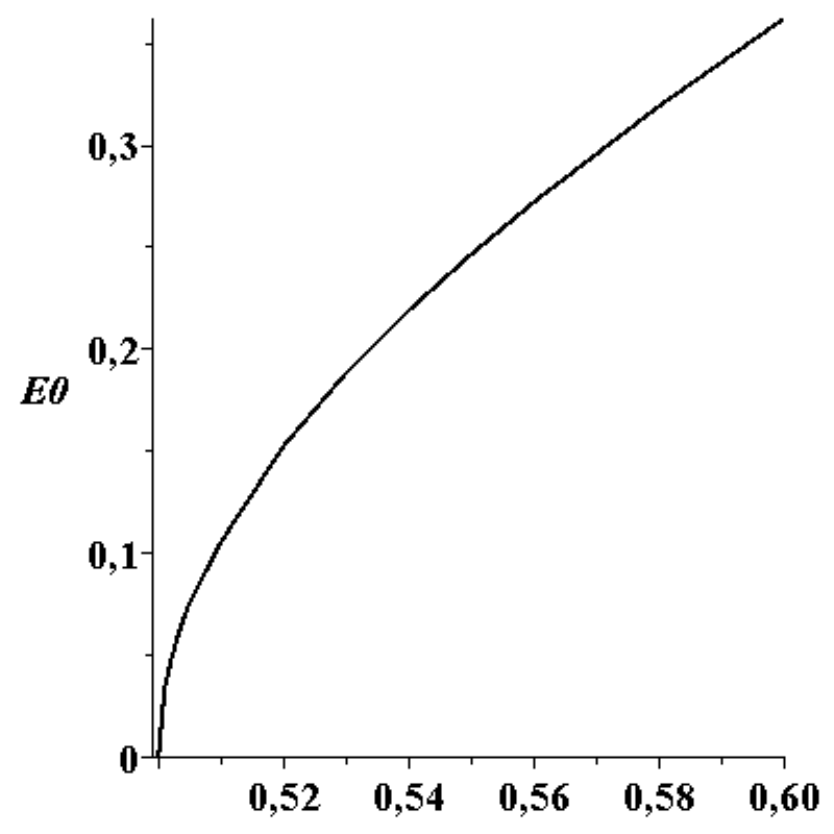}
\end{minipage}
\hspace{1.5cm}
\begin{minipage}{0.4\columnwidth}
\includegraphics[scale=0.3]{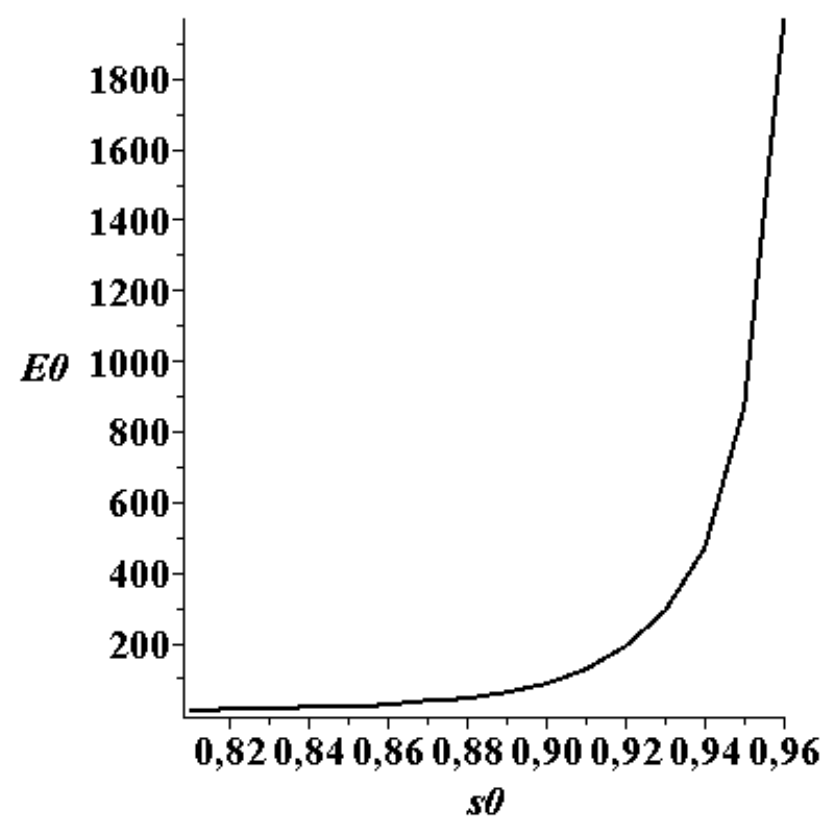}
\end{minipage}
\caption{ $k=2$, $V_0=0$, $\nu_0=1$. The threshold curve close to $s_0=0.5$ (left) and close to $s_0=1$ (right) on the plane $(s_0, E_0))$ for $E_0>0$. The blow-up domain is under the curve.  The picture is symmetric with respect to $E_0=0$.  }\label{Pic1}
\end{figure}
\end{center}

Fig.2 shows the decay of the amplitude of $E$ with time for large $E_0$ (the behavior of $V$ is similar, but the graph starts from zero), and the respective behavior of $Q(t)$. Here $E_0=85$, $s_0=0.9$. Note that the time of the blow-up increases with $E_0$. Indeed, large $E_0$ acts as a large initial damping.

\begin{center}

\begin{figure}[htb]

\hspace{-1cm}
\hspace{1cm}
\begin{minipage}{0.4\columnwidth}
\includegraphics[scale=0.3]{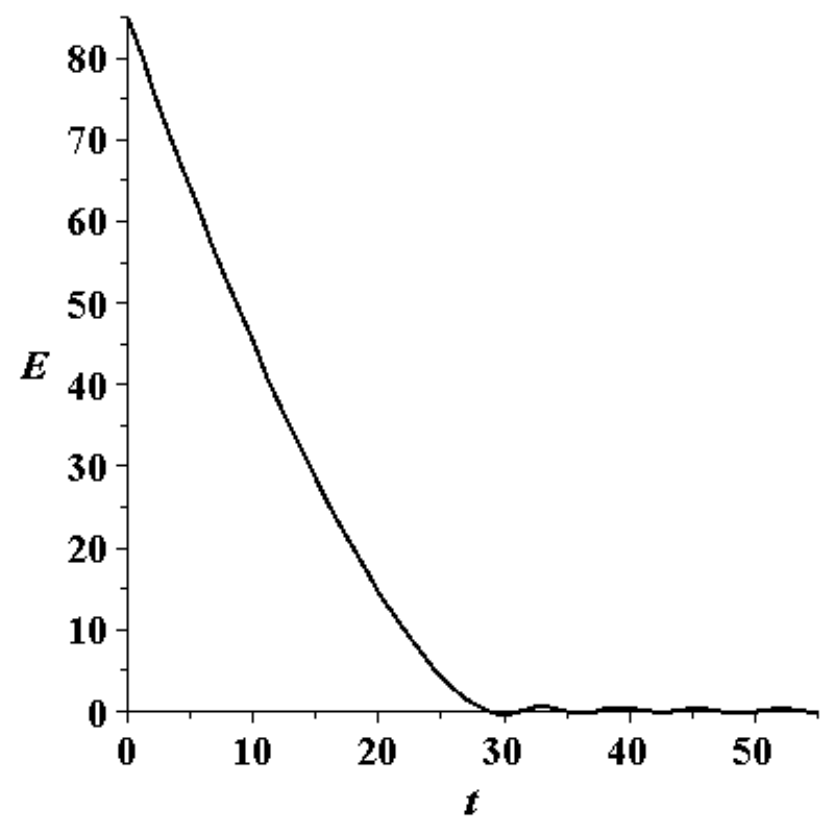}
\end{minipage}
\hspace{1.5cm}
\begin{minipage}{0.4\columnwidth}
\includegraphics[scale=0.3]{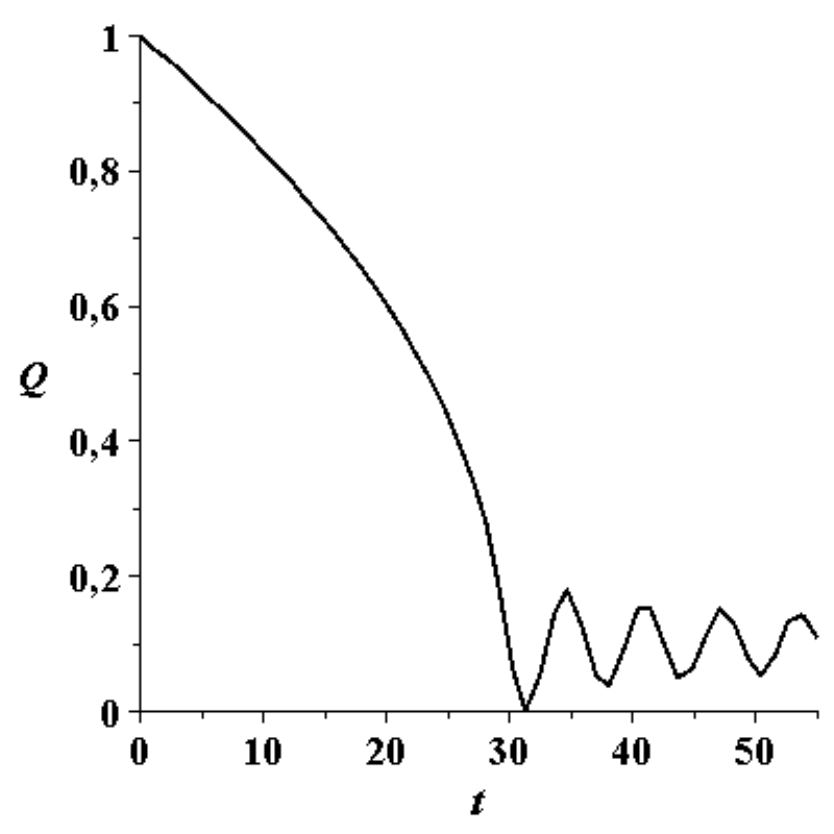}
\end{minipage}
\caption{$k=2$, $V_0=0$, $E_0=85$, $s_0=0.9$, $\nu_0=1$. The behavior of $E$ (left) and $Q$ (right). }\label{Pic2}
\end{figure}
\end{center}


\section{Radially symmetric multidimensional case}\label{MultiD}

Let us consider  radially symmetric solutions of  \eqref{4}
\begin{eqnarray*}\label{sol_form}
\bV=F(t,r)\br, \quad \bE=G(t,r)\br,
\end{eqnarray*}
where $\br = (x_1,x_2,...,x_\bd)$ is the radius-vector, $r=\sqrt{x_1^2+x_2^2+...+x_\bd^2}$.

The initial data that correspond to these solutions are
\begin{equation}\label{CD1}
(\bv, \bE) |_{t=0}=(\bv_0(r), \bE_0(r))= (F_0(r) {\bf r}, G_0(r) {\bf r} ), \quad (F_0(r)
, G_0(r) ) \in C^2(\bar {\mathbb R}_+).
\end{equation}

A natural question arises about the class of initial data that lead to a blow-up of the solution in the case of a velocity-dependent damping coefficient for radially symmetric oscillations in a space of dimension $\bd>1 $.
It is known that, in the absence of damping in dimensions other than $1$ and $4$, there is no neighborhood of zero  such that the solution with initial data in this neighborhood remains smooth for all $t$ \cite{R22_Rad}, \cite{R_Doping}. For $\bd = 1$, such a neighborhood is given by \eqref{qs}; for $\bd  = 4$, it has a much more complex form and is found in \cite{R_Exeptional}. For constant coefficient of damping, it is known that such a neighborhood of zero exists, expands with increasing of the  coefficient of damping, and solutions with initial data in this neighborhood exponentially approach zero over time \cite{RD_Radial_EP_friction}. This is a natural result.
The effect of damping can be investigated numerically using Radon's lemma, as in Sec. 2. This requires solving a system of six nonlinear differential equations describing the solution along a characteristic, starting from point $r_0$. For the convenience of the reader we give a sketch of the reasonings \cite{R22_Rad}, \cite{RD_Radial_EP_friction}, applicable to the new situation.

If $\bV=F(t,r)\br$, $ {\bf E}=G(t,r)\br$,  $r_0\in \bar{\mathbb R}_+$, 
we can introduce scalar functions $\mathcal V=Fr$ and  $\mathcal E=Gr$. Then \eqref{FG} implies
 \begin{eqnarray}
\label{VE}
&& \dot {\mathcal V}=-\mathcal E-\nu(\mathcal V) \mathcal V, \quad \dot {\mathcal E}=\mathcal V+(1-\bd) \frac{\mathcal V \mathcal E}{r}, \quad \dot{r}=\mathcal V.
 \end{eqnarray}
and
\begin{eqnarray}
\label{FG}
&& \dot F=-F^2-G-\nu(\mathcal V) F, \quad \dot G=F-dFG, \quad \dot{r}=Fr, \\
&& r(0)=r_0>0,\, F(0)=\frac{V_0(r_0)}{r_0},\, G(0)=\frac{E_0(r_0)}{r_0}.\nonumber
 \end{eqnarray}

Let us denote $D=\Div\bV,$ $ \lambda=\Div{\bf  E}$.
Equations \eqref{4} imply
\begin{equation*}
\frac{\partial D}{\partial t}+(\bV \cdot \nabla D)=-D^2+2(\bd-1) FD- \bd(\bd-1) F^2-\lambda -a(\mathcal V) D,$$ $$\frac{\partial \lambda}{\partial t}+(\bV \cdot \nabla \lambda)= D(1-\lambda).
\end{equation*}
Along the characteristics given as  $\dot{r}=Fr$  the functions $D, \lambda$ obey
\begin{equation}\label{Dlam}
\dot D=-D^2+2(\bd-1) FD- \bd(\bd-1) F^2-\lambda -a(\mathcal V) D, \quad
\dot \lambda= D(1-\lambda).
\end{equation}
We introduce new variables $ u=D-\bd F, v=\lambda-\bd G$. Systems \eqref{Dlam} and
\eqref{FG} imply
\begin{eqnarray}
\label{uv}
\dot u =-u^2-2uF-v -a(\mathcal V) u,\quad \dot v=-uv+(1-\bd G)u-\bd F v.
\end{eqnarray}

System \eqref{uv} can be linearized my means of the Radon lemma (e.g. \cite{F}, \cite{Radon}).

$(p_1(t),p_2(t))$ and  $Q(t)$ solves the linear system
\begin{eqnarray}\label{LinQ4}
\left(\begin{array}{c}  \dot Q\\  \dot p_1\\\dot p_2
  \end{array} \right)=M \left(\begin{array}{c}  Q\\  p_1\\p_2
  \end{array} \right), \quad  M=\begin{pmatrix}
   0 & 1 & 0 \\
  0&-2F- a(Fr)  & -1\\
  0&1-dG & - dF\\
\end{pmatrix}
 \end{eqnarray}
  subject to the initial data
\begin{eqnarray*}
\left(\begin{array}{c}
  Q\\  p_1\\p_2
  \end{array} \right)(0)=
  \left(\begin{array}{c}
   1\\  \Div\bV_0-\bd F_0(r_0)\\\Div  {\bf  E}_0-\bd G_0(r_0)
  \end{array} \right).
  \end{eqnarray*}
  Since
 $ u(t)=p_1(t) Q^{-1}(t)$, $ v(t)=p_2(t) Q^{-1}(t)$, To find a condition for the derivatives to become infinite (or remain bounded), we need to find conditions under which the component $Q(t)$ vanishes for some $t>0$ (or, correspondingly, does not vanish for any $t$).
First, we note that the amplitude of the oscillations described by the  system  \eqref{FG}, 
decreases (see the proof of this fact in \cite{RD_Radial_EP_friction}).

However, the numerical results turned out to be quite unexpected: for $\bd\ne 1$ and $\bd\ne 4$, any small perturbation of the zero equilibrium loses smoothness over a finite time. Specifically, the amplitude of the $Q$ oscillations increases infinitely, and, starting from $1$, $Q$ vanishes for some $t>0$.
In other words, the presence of velocity-dependent damping \eqref{nu} has no effect on improving the smoothness properties of the solution. The time required for a blow up decreases as the dimension $\bd$ and the exponent $k$ increase.

For $\bd= 4$, the situation is different and resembles the case $\bd= 1$, i.e., the domain of the initial data corresponding to globally smooth solutions expands in the presence of damping.

However, the fact that velocity-dependent power-law damping does not prevent a blow-up can be explained as follows. If, by analogy with the one-dimensional case,  the oscillation amplitude decreases according to a power law, then the resulting power-law factor in the asymptotics of the derivatives (see \eqref{1Dqas}) cannot damp oscillations, whose amplitude grows exponentially according to Floquet theory in the case of instability. Below we concretize this idea. 

\subsection{Blow up of small oscillations, $\bd\ne 1,$  $\bd\ne  4$}

\begin{theorem}\label{MT} For  $\bd\ne 1$ and $\bd\ne  4$ the solution of the Cauchy problem
\eqref{4}, \eqref{CD1}, with $\nu(|V|)= \nu_0|V|^k$, $k=2m$, $m\in \mathbb N$, blows up in a finite time for
all arbitrary small general perturbation of nontrivial initial  data.
\end{theorem}

First of all we prove the analog of Lemma \ref{L1} for the multidimensional case.

Let us recall that the first integral of \eqref{FG} at $\nu=0$ is known \cite{R22_Rad}.
They are
\begin{eqnarray}\label{Int_d}
&&\Phi(G,F)=\frac{(\bd-2) F^2-2G+1}{(\bd-2)(1-\bd G)^{\frac{2}{\bd}}}=\Phi(G_0,F_0)=C_\bd,\\&& C_\bd=\frac{(\bd -2) F^2_0-2G_0+1}{(\bd-2)(1-\bd G_0)^{\frac{2}{\bd}}},\nonumber
 \end{eqnarray}
for $\bd\neq 2$, and
\begin{eqnarray}\label{Int_2}
&&\Phi(G,F)=\frac{2F^2+\ln(1-2G)(1-2G)+1}{2(1-2G)}=\Phi(G_0,F_0)=C_2, \\ &&C_2= \frac{2F^2_0+\ln(1-2G_0)(1-2G_0)+1}{2(1-2G_0)},\nonumber
\end{eqnarray}
 for $\bd=2$. 
 The curves given as \eqref{Int_2} and \eqref{Int_d} are bounded, they contain  the origin and intersect the axis $F=0$ in two points: $(G_-,0)$, $G_-<0$, and $(G_+,0)$, $G_+>0$, see the pictures in \cite{R22_Rad}.
Further, \eqref{FG} implies
\begin{equation}\label{rG}
1-\bd\,G=C\, r^{-\bd}, \quad C=C(r_0)={\rm const}>0
\end{equation}
therefore,  expressing $r$ from here  and taking into account $\bV=F(t,r)\br$ we get
\begin{equation*}\label{bV}
|\bV|^2=\frac{1}{\bar C^2}\,
\frac{ F^2}{(1-\bd G)^{\frac{2}{\bd}}},\quad \bar C=C^{-\frac{1}{\bd}}.
\end{equation*}

\begin{lem}\label{L1md}  If  $\nu(|\bV|)=- \nu_0|\bV|^k$, then along the characteristic $r=r(t)$, starting from $r_0\in \mathbb R$
 \begin{equation*}\label{1DV_tmd}
  |\bV|^2 \le \frac{1}{\bar C^2}\, \left((\Phi(G_0(r_0),F_0(r_0)))^{-\frac{k}{2}} + {\bar\nu_0 k} t\right)^{-\frac{2}{k}},
 \end{equation*}
 \begin{equation*}\label{1DEV_tmd}
 \Phi(G(r_0),F(r_0))\ge  \Phi(G(r(t)),F(r(t)))\ge \left((\Phi(G_0(r_0),F_0(r_0)))^{-\frac{k}{2}} + {\bar\nu_0 k} t\right)^{-\frac{2}{k}},
 \end{equation*}
where $\bar\nu_0=\nu_0 \bar C^{-\frac{1}{k}}$.
\end{lem}

 \proof
 Note that
 \begin{equation*}
 \Phi=\bar C^2|\bV|^2+S,
 \end{equation*}
 where $\bar C^2\,|\bV|^2=\frac{F^2 }{(1-\bd G)^{\frac{2}{\bd}}}$, $S=\frac{1-2G}{(\bd-2)(1-\bd G)^{\frac{2}{\bd}}}$ for $\bd>2$ and $S=\frac12 \ln(1-2G)+\frac{1}{2(1-2G)}$, for $\bd=2$. It is easy to see that taking into account $G<\frac{1}{\bd}$, see \eqref{rG}, in both cases $S> 0$.

  From \eqref{FG} we have
\begin{eqnarray*}\label{Lyap}
\frac{\mathrm d \Phi}{\mathrm d t}=-\frac{2\nu(|\bV|) F^2 }{(1-\bd G)^{\frac{2}{\bd}}}\le 0
\end{eqnarray*}
or
\begin{equation*}\label{EV1md}
  \dfrac{d}{dt} (\bar C^2 |\bV|^2+{S})= -2\nu(|\bV|)\,\bar C^2\, \bV^2= -2 \nu_0\,|\bV|^k\,\bar C^2\, \bV^2.
 \end{equation*}
 Further reasoning repeats the proof of Lemma \ref{L1}.
 $\Box$

 \begin{lem}{\rm (Long time asymptotics of $\mathcal V$)} If  $\nu(|\bV|)= \nu_0|\bV|^k$, $k=2m$, $m\in \mathbb N$, and the perturbations of equilibrium of system \eqref{VE} are small, then along the characteristic $r=r(t)$, starting from $r_0\in \mathbb R$ the following asymptotics takes place
 \begin{equation*}\label{mDVas}
  \mathcal V(r(t))\sim t^{-\frac{1}{k}}\,\sum\limits_{n=0}^\infty \Psi_m(t) t^{-n},\quad t\to\infty,
   \end{equation*}
   where $\Psi_n$ are periodic, $n\in\mathbb N$.
 \end{lem}

 \proof
 From \eqref{VE} we obtain
 \begin{equation*}
  \ddot {\mathcal V}+ \nu_0 (k+1) | {\mathcal V}|^k \dot  {\mathcal V}+ {\mathcal V}+(1-\bd) \mathcal V G=0.
   \end{equation*}

 We replace $\mathcal V=t^{-\frac{1}{k}}\,v(t),$ with an analytic in the neighborhood of infinity $v(t)$ and obtain
  \begin{equation*}\label{1Dudd}
  \ddot v+ v+\left(\nu_0 (k+1) v^k-\frac{2}{k}\right)\frac{\dot v}{t}-\left(\nu_0 (k+1) v^k-\frac{k+1}{k^2}\right)\frac{ v}{t^2}+(1-\bd) v G
 =0.
   \end{equation*}
As in the proof of Theorem 1, we assume that $v=\sum\limits_{n=0}^\infty v_n(t) t^{-n} $, where $v_n(t)=\sum\limits_{l=0}^n \Phi_{n,l}(t) t^l $, where   $\Phi_{n,l}(t)$ are periodic. As before, this means that $v=\sum\limits_{n=0}^\infty \Psi_n(t) t^{-n}$, where $\Psi_0= \sum\limits_{j=1}^\infty \Phi_{j,j}$, $\Psi_1= \sum\limits_{j=2}^\infty \Phi_{j,j-1}$,...,$\Psi_n= \sum\limits_{j=n-1}^\infty\Phi_{j,j-n+1}$.

Substitution gives the following equations for $v_n$, $n=0,1,\dots$:

  \begin{eqnarray}
 && \ddot v_0+ v_0+(1-\bd) v_0 G =0,\label{v0}\\
 && \ddot v_1+ v_1+(1-\bd) v_1 G=\frac{2}{k} \dot v_0 -\nu_0(k+1)\dot v_0 \ddot v_0,\nonumber\\
 && \cdots\nonumber\\
 && \ddot v_n+ v_n+(1-\bd) v_n G=\frac{2(n-1)+1}{k} \dot v_{n-1} +P_2(v_j,\dot v_j, \ddot v_j), \quad j=0,\dots, n-1,\nonumber
   \end{eqnarray}
where $P_2(v_j,\dot v_j, \ddot v_j)$ is a second order polynomial of $v_0,\dot v_0, \ddot v_0, \dots, v_{n-1},\dot v_{n-1}, \ddot v_{n-1}$.
Equation \eqref{v0} is a corollary of the system
\eqref{FG}, which solution is periodic (see \eqref{Int_d}, \eqref{Int_2}). However, if we consider small perturbations of the equilibrium $v_0=0$, we take into account the linear terms only and the solution corresponds to the case $\bd=1$, considered in Theorem 1, i.e.
$v_0=A \sin (t+B)$, $A, B$ are arbitrary constants. The asymptotics coincides with one in the case $\bd=1$, i.e. \eqref{1DVas}.  $\Box$

\medskip

{\it Proof of Theorem 4.}  System \eqref{LinQ4} implies
\begin{eqnarray*}
Q(t)=1+\int\limits_0^t p_1(\tau) d\tau, \quad \dot p_1 =-(2F+a(\mathcal V))p_1-p_2,\quad \dot p_2=(1-\bd G)p_1-\bd Fp_2.
\end{eqnarray*}
Together with \eqref{FG} it follows
\begin{eqnarray*}
 \ddot p_1+((\bd+2) F+a(\mathcal V))\dot p_1+(2(\bd-1)F^2-(2+\bd)G+(\bd a(\mathcal V) -2 \nu(\mathcal V))\nu a'(\mathcal V) \dot{\mathcal V} +1)p_1=0.\end{eqnarray*}
We change
\begin{eqnarray}\label{p1}
p_1(t)=H(t)e^{-\frac{1}{2} \int\limits_0^t a(\mathcal V(\tau)) d\tau}
e^{-\frac{\bd+2}{2} \int\limits_0^t F(\tau) d\tau}
\end{eqnarray}
and obtain
\begin{eqnarray}\label{H}
\ddot H +JH=0,\end{eqnarray} with
\begin{eqnarray*}\label{J}
J=1-\frac{(\bd+2)}{2}G-\frac{(\bd-2)(\bd-4)}{4}F^2+\frac{\bd+6}{2}(a(\mathcal V)-\nu(\mathcal V))  F-\frac{1}{4}a^2(\mathcal V)+\frac{1}{2}a'(\mathcal V)\dot{\mathcal V}.
\end{eqnarray*}
Now we take into account
$\nu(\mathcal V)=\nu_0  | {\mathcal V}|^k$,
$a(\mathcal V)=\nu_0 (k+1) | {\mathcal V}|^k$, \eqref{VE}, \eqref{rG} and see that if the initial perturbation of the zero equilibrium is assumed small then
system  \eqref{FG}
implies that if the initial deviations from the zero equilibrium are
of order $\epsilon$, $\epsilon\to 0$, then $G=\epsilon\,\cos
t+o(\epsilon)$, $F=\epsilon\,\sin t+o(\epsilon)$. Therefore  $\mathcal V=Fr=\epsilon\,\sin t+o(\epsilon)$, and
$$J=1-\frac{d+2}{2}\,\epsilon\, \cos t +o(\epsilon),\quad \epsilon\to 0$$
and up to $o(\epsilon)$ the function $J$ is periodic. Let us denote
$$\bar J=1-\frac{d+2}{2}\,\epsilon\, \cos t$$
 The Floquet
 theorem  implies that for any periodic  $J(t)$ with period $T$, any solution of  \eqref{H} has
the form $P={\mathcal P}_1(t) e^{\sigma_1 t} + {\mathcal P}_2(t)
e^{\sigma_2 t}$ or $P=e^{\sigma t} ({\mathcal P}_1 (t)+ {\mathcal
P}_2(t)\, t)$, where the functions ${\mathcal P}_1, {\mathcal P}_2$
are periodic with period
   $T$. Since in our case $\bar J(t)=\bar J (-t)$, therefore \eqref{H} has solutions $e^{\mu t} \mathcal P(t)$, and $e^{-\mu t} \mathcal P(-t)$, $\mathcal P$
 is $T$-periodic (up to $o(\epsilon)$).
As is shown in \cite{R22_Rad},  the
 characteristic exponent $\mu$ for $J=\bar J$ in  the cases $\bd\ne  1$, $\bd\ne  4$  is real. Thus, the amplitude of oscillations of $H$ rises exponentially.

Further,
$e^{-\frac{\bd+2}{2} \int\limits_0^t F(\tau) d\tau}$ is bounded.   Lemma 3 implies that for small $\epsilon$ we can use obtained in the proof of Theorem 2 expansion  \eqref{aVas}  of $a(\mathcal V)$. Therefore
$$e^{-\frac{1}{2} \int\limits_0^t a(\mathcal V(\tau)) d\tau}\sim t^{-\varkappa\epsilon^k},\quad \varkappa>0, \quad t\to\infty,$$
$\varkappa={\nu_0 a_k}$,  $a_k=\,\frac{(k+1)!!}{k!!}$.
Thus, the decrease of the  first multiplier in \eqref{p1} is not enough to neutralize the exponential rise of $H$, and therefore   $p_1$ rises exponentially as well.
Thus, the component
$
Q(t)=1+\int\limits_0^t p_1(\tau) d\tau$ at some $t>0$ turns into zero and the
 derivatives blow up for arbitrary small $\epsilon>0$.  $\Box$

\bigskip

\begin{remark}
 The phenomenon of decreasing the blowup time with increasing of  the dimension $\bd$ and the exponent $k$ can be also explained.
 Indeed, for larger $k$ the multiplier $t^{-{\nu_0 a_k}\epsilon^k}$ decreases slower. Whereas the multiplier $\mu$, responsible for the rise of the amplitude of oscillations,  is proportional to $(\bd+2)^4 \epsilon^4 $ (see \cite{R22_Rad}).
\end{remark}
\begin{remark}
The case $\bd=  4$ is qualitatively similar to the case $\bd=  1$, considered in Sec.2. In other words, there exists a neighborhood of the
 zero steady state such that the solution with the data from this neighborhood keep smoothness for all $t>0$, and this neighborhood increases in the space $(F_0, G_0, u_0, v_0)$ with the coefficient of friction. Nevertheless, the structure of the threshold hypersurface
 is very complicated even for $\nu=0$ \cite{R_Exeptional}, therefore in the case with damping \eqref{nu} can hardly by obtained analytically.
\end{remark}

\begin{remark}
Note that for $\nu=0$ in the multidimensional radially symmetric case, there exist simple wave solutions for which $F=F(G)$. The initial data corresponding to such solutions forms a set of measure zero in the  four-dimensional space of initial data. The properties of such solutions differ from those of the general solutions discussed above. Specifically, among such solutions in our case, as in the case $\nu=0$, there may be nontrivial globally smooth solutions \cite{R22_Rad}. However, affine solutions ($F=F(t)$, $G=G(t)$), the simplest of those belonging to this class, are absent in our case.
\end{remark}

\section{Discussion}\label{Disc}
This paper investigates the influence of a power-law velocity-dependent damping coefficient on the repulsive Euler-Poisson equations. The ability of solutions to maintain smoothness over infinite time is examined. One-dimensional and multidimensional radially symmetric cases are analyzed analytically and numerically. It is found that for $\bd=1$ and $\bd=4$, the domain of initial data corresponding to a smooth solution expands, whereas for other dimensions, the presence of damping has no effect on this domain.


From a mathematical perspective, it makes sense to study velocity-dependent damping in a different form. This could be, for example,
$\nu= \nu_0 \exp(\lambda |V|^k)$ or $\nu=\nu_0 \ln (1+|V|^k)$, $\nu_0>0$, $\lambda, k \in \mathbb R$.
Such an assumption can lead to qualitatively different phenomena in both one-dimensional and multidimensional cases.
Moreover, the restriction $k=2m$, $m\in \mathbb R$, is due to the possibility of obtaining asymptotics \eqref{1DVas}, \eqref{1Dqas}.
Numerical calculations show that this restriction does not affect either the qualitative properties of the solutions or the equation for the threshold curve \eqref{deleps}.

From a physical point of view, it is interesting to consider the damping coefficient, which depends on the kinetic energy of the particles, that is, on the combination $n |V|^2$. 

\section*{Acknowledgments}

Supported  by the Russian Science Foundation under grant no. 23-11-00056P,
performed at RUDN University.

  \end{document}